\documentclass{article}

\usepackage{comment}
\usepackage[utf8]{inputenc}
\usepackage[T1]{fontenc}
\usepackage{amssymb}
\usepackage{amsmath, amsthm}
\usepackage{amsfonts}
\usepackage{graphics}
\usepackage{color}
\usepackage{xcolor}
\usepackage{graphicx}
\usepackage{graphics}
\usepackage{color}
\usepackage{amsmath}
\usepackage{amsfonts}
\usepackage{amssymb}
\usepackage{enumerate}
\usepackage{hyperref}
\usepackage[english]{babel}
\usepackage{bm}

\newtheorem{thm}{Theorem}

\newtheorem{cl}{Claim}
\newtheorem{Lemma}{Lemma}

 \title{Augmenting a hypergraph to have a matroid-based $(f,g)$-bounded $(\alpha,\beta)$-limited packing of rooted hypertrees}
\author{Pierre Hoppenot, Zolt\'an Szigeti\\ University Grenoble Alpes, INP, G-SCOP}

\begin{document}

\maketitle

\begin{abstract}
The aim of this paper is to further develop the theory of packing trees in a graph. We first prove the classic result of Nash-Williams \cite{NW} and Tutte \cite{Tu} on packing spanning trees by adapting Lov\'asz' proof \cite{Lov} of the seminal result of Edmonds \cite{Egy} on packing spanning arborescences in a digraph. Our main result on graphs extends the theorem of Katoh and Tanigawa \cite{KT} on matroid-based packing of rooted trees by characterizing the existence of such a packing satisfying the following further conditions: for every vertex $v$, there are a lower bound $f(v)$ and an upper bound $g(v)$ on the number of  trees rooted at  $v$ and there are a lower bound $\alpha$ and an upper bound $\beta$ on the total number of roots. We also answer the hypergraphic version of the problem. Furthermore, we are able to solve the augmentation version of the latter problem, where the goal is to add a minimum number of edges to have such a packing. The methods developed in this paper to solve these problems may have other applications in the future.
\end{abstract}

\section{Introduction}

The first major result on packing spanning trees is due to Nash-Williams~\cite{NW} and Tutte~\cite{Tu}. They independently characterized graphs having a packing of $k$ spanning trees; in other words $k$ pairwise edge-disjoint spanning trees. As a first contribution of this paper we provide a new proof of their result. We believe that the proof is new but we know that the approach is old. Actually, Lov\'asz \cite{Lov} provided an elegant and simple proof of Edmonds' result on packing spanning arborescences and here we work out how the same idea can be applied in the undirected case.

Since then, the result of Nash-Williams~\cite{NW} and Tutte~\cite{Tu} has been extended in several ways. Notably, Katoh and Tanigawa \cite{KT} characterized graphs admitting a complete matroid-based packing of rooted trees, see Theorem \ref{thmKT}. Here the rooted trees are not necessarily spanning. However, a matroid is given on the root set and the packing must satisfy a matroid constraint, informally meaning that  every vertex is reachable from a basis of the matroid in the rooted trees of the packing. Katoh and Tanigawa  explain in \cite{KT}  an interesting application of this theorem in rigidity theory. 

Our goal is to extend the result of Katoh and Tanigawa \cite{KT} on matroid-based packing of rooted trees. To do so, we develop useful tools mainly based on our improved knowledge of the uncrossing of two partitions of a set. First, we  show the submodularity of some functions with two variables. We also give a tool which shows that it is possible to simultaneously cover, with an edge set, two supermodular functions on partitions of a vertex set, see Theorem \ref{coveringpartsuper12}. This has been proved in a special case in \cite{HS1}. Likewise, we develop a tool which shows that it is possible to trim a hypergraph that covers two supermodular functions on partitions to a graph that covers the same functions,  see Theorem \ref{trimminglemma}. 

The discovery of the submodularity of the above mentioned functions on partitions allows us to give the rank function of a new matroid which was inspired by the work of Katoh and Tanigawa \cite{KT}. More precisely, for a graph $G = (V, E)$,  a multiset $S$ of vertices in $V$,  and a matroid ${\sf M} = (S, r_{\sf M}),$ we give a matroid whose independent sets of size $r_{\sf M}(S)|V|$ are exactly the sets $F\cup R,$ where $F$ is the edge set and $R$ is the root set of an ${\sf M}$-based packing of rooted trees in $G$, see Theorem \ref{thmKTimplicitgen}. 

This matroid along with another matroid, the bounded direct sum of matroids (see Theorem \ref{genpart}), play a crucial role in the solution of the following problem. Given  a graph $G=(V,E)$, a multiset $S$ of vertices in $V$, $k\in\mathbb Z_+,$ $f,g: V\rightarrow \mathbb Z_+$ functions, and ${\sf M}=(S,r_{\sf M})$ a matroid, characterize the existence of an {\sf M}-based $(f,g)$-bounded packing of $k$ rooted trees  in $G$, where $(f,g)$-bounded means that every vertex $v\in V$ is the root of at least $f(v)$ and at most $g(v)$ rooted trees, see Theorem \ref{hvkgcvkl}. We extend this result to $(\alpha, \beta)$-limited packings, meaning that the given value of the number of roots is relaxed to an interval $[\alpha, \beta]$, see Theorem \ref{hvkgcvkl2}. 
Using the newly found submodularity of some functions on partitions and the previously mentioned result on trimming, we are then able to generalize the previous result to get a characterization of  hypergraphs having an ${\sf M}$-based $(f, g)$-bounded $(\alpha, \beta)$-limited packing of rooted hypertrees, see Theorem \ref{hvkgcvkl2hyp}. This extends an earlier theorem of Frank, Kir\'aly and Kriesell~\cite{fkk} that generalized Nash-Williams and Tutte's theorem to hypergraphs. For further new results on packing hypertrees, see \cite{HMSz}. 

We are also able to formulate and prove the conditions under which a hypergraph can be augmented (in term of minimum number of edges) to contain  an ${\sf M}$-based $(f, g)$-bounded $(\alpha, \beta)$-limited packing of rooted hypertrees, see Theorem \ref{kjvjhcychjkjhyp}. The readers interested  in similar augmentation problems are invited to see \cite{HS1} and \cite{szighypbranc}.

The organization of this paper is as follows. 
In Section \ref{definitions} we provide the necessary definitions.
In Section \ref{partitions} we first introduce  some  submodular and some supermodular functions on partitions. Then we prove our tools on covering and trimming about supermodular functions on partitions.
Section \ref{packres} contains the above mentioned results on packing trees and their proofs.

\section{Definitions}\label{definitions}

We denote by $\pmb{\mathbb{Z}}$ the set of integers and $\pmb{\mathbb{Z}_+}$ the set of non-negative integers. Let $\bm{V}$ be a finite set. For a function $m : V \rightarrow \mathbb{Z}$ and a subset $X$ of $V,$ we define $\bm{m(X)} = \sum_{x \in X} m(x)$.
For $X \subseteq V$, we denote by $\bm{\overline{X}}$ its \textit{complement}, that is $V\setminus X$. We say that $X$ \textit{separates} two distinct elements  of $V$ if $X$ contains one of them and $\overline{X}$ contains the other one. A \textit{multiset} of $V$ is a set of elements of $V$ allowing multiplicities. For a multiset $S$ of $V$ and $X \subseteq V$, $\bm{S_X}$ denotes the multiset of $V$ consisting of the restriction of $S$ on $X$. A set $\mathcal{S}$ of subsets of $V$ is called a \textit{family} if the subsets of $V$ are taken with multiplicities in $\mathcal{S}$. For a family $\mathcal{S}$ of subsets of $V$ and a subset $X$ of $V$, we denote by $\bm{\mathcal{S}_X}$ the family containing the sets in $\mathcal{S}$ that intersect $X$. Two subsets $X$ and $Y$ of $V$ are called \textit{properly intersecting} if none of $X \cap Y$, $X\setminus Y$, and $Y\setminus  X$ is empty. The operation that replaces two properly intersecting sets by their intersection and their union is called \textit{uncrossing}. For a family $\mathcal{F}$ of subsets of $V$, the \textit{uncrossing method} consists in applying repetitively  the uncrossing operation as long as properly intersecting sets exist in the family.

A set of pairwise disjoint subsets of $V$ such that their union is $V$ is called a \textit{partition} of $V$. 
We say that a subset $X$ of $V$ \textit{crosses} a partition $\mathcal{P}$ of $V$ if $X$ intersects at least two members of $\mathcal{P}$.
Let $\mathcal{P}_1$ and $\mathcal{P}_2$ be two partitions of $V$ and $\mathcal{P} = \mathcal{P}_1 \cup \mathcal{P}_2$. We use the uncrossing method on the family $\mathcal{P}$ to obtain a new family $\mathcal{P}'$ which contains no properly intersecting sets. Taking respectively the minimal and maximal sets in $\mathcal{P}'$, we obtain two partitions $\mathcal{P}_1'$ and $\mathcal{P}_2'$ of $V.$ We call $\mathcal{P}_1'$  the \textit{intersection}  of $\mathcal{P}_1$ and $\mathcal{P}_2$, and we denote it by $\bm{\mathcal{P}_1 \sqcap \mathcal{P}_2}$;  we call $\mathcal{P}_2'$  the \textit{union}  of $\mathcal{P}_1$ and $\mathcal{P}_2$ and we denote it by $\bm{\mathcal{P}_1 \sqcup \mathcal{P}_2}$. We mention that while $\mathcal{P}_1 \sqcap \mathcal{P}_2$ depends on the choices during execution of the uncrossing method, $\mathcal{P}_1 \sqcup \mathcal{P}_2$ is uniquely defined.

Let $S$ be a finite ground set. A set function $b$ on $S$ is called \textit{non-decreasing} if $b(X) \le b(Y)$ for all $X\subseteq Y \subseteq V$ and \textit{subcardinal} if $b(X) \le |X|$ for every $X \subseteq V$. We say that $b$ is \textit{submodular} (resp. \textit{intersecting submodular}) if $b(X) + b(Y) \ge b(X \cap Y) + b(X \cup Y)$ for every sets (resp. properly intersecting sets) $X, Y \subseteq V$. A set function $p$ on $S$ is called \textit{supermodular} if $-p$ is submodular. A set function $m$ on $S$ is called \textit{modular} if it is submodular and supermodular. Let $r$ be a non-negative, integer-valued, non-decreasing, subcardinal and submodular set function on $S$. Then $\bm{{\sf M}} = (S, r)$ is called a \textit{matroid} and $r$ is called the \textit{rank function} of ${\sf M}$. A subset $X$ of $S$ is called an \textit{independent set} of ${\sf M}$ if $r(X) = |X|$. The set of independent sets of ${\sf M}$ is denoted by $\bm{\mathcal{I}_{{\sf M}}}$. A maximal independent set of ${\sf M}$ is called a \textit{basis} of {\sf M}. If $S$ is a basis of ${\sf M}$, then  ${\sf M}$ is called  \textit{free matroid}. If the bases are all the subsets of $S$ of size $k$, then {\sf M} is called a {\it uniform matroid} of rank $k.$ For $T\subseteq S,$   the matroid $\bm{{\sf M}|_T}=(T,r)$, obtained from ${\sf M}$ by deleting the elements $S\setminus T$, is called \textit{restricted} matroid on $T.$ For an independent set $X$ in ${\sf M}$, the  matroid $\bm{{\sf M}{/X}}=(S\setminus X,r_{/X})$, whose set of independent sets is $\{Y\subseteq S\setminus X: X\cup Y\in \mathcal{I}_{\sf M}\}$ and whose rank function is $\bm{r_{/X}}(Z)=r(X\cup Z)-|X|$ for every $Z\subseteq S\setminus X$, is called \textit{contracted matroid}.

Let $\bm{G} = (V, E)$ be a graph with \textit{vertex set} $\bm{V}$ and \textit{edge set} $\bm{E}$. A graph $G'$ is a \textit{subgraph} of $G$ if it is obtained from $G$ by deleting some vertices and some edges. Furthermore, if $V(G')=V(G)$, then $G'$ is called a \textit{spanning} subgraph of $G$. Let $X$ be a subset of $V.$ We denote by $\bm{i_E(X)}$ the number of edges of $E$ in $X.$ It is known that $i_E$ is supermodular. We denote by $\bm{G[X]}$ the subgraph of $G$ after deleting $\overline X.$ We denote by $\bm{G/X}$ the graph obtained from $G$ by contracting $X$, that is by replacing $X$ by a new vertex $\bm{v_X}$, by deleting all the edges in $X$, and replacing every edge $uv$ in $E$ such that $v\in X$ and $u\in\overline X$ by an edge $v_Xu.$ For disjoint $X,Y\subseteq V,$ $\bm{d_E(X,Y)}$ denotes the number of edges $xy$ in $E$ with $x\in X$ and $y\in Y.$ A \textit{forest} of $G$ is a subgraph of $G$ that contains no cycle. A \textit{tree} is a connected forest. A couple $(S, F)$ is a \textit{rooted forest} of $G$ if $F$ is a forest of $G$ and $S$ is a set containing exactly one vertex of each connected component of $F$. The set $S$ is called the \textit{root set} of the rooted forest. 

By a \textit{packing} of rooted forests in $G$, we mean a set $\mathcal{B}$ of rooted forests of $G$  that   are edge disjoint. For two functions $f, g : V \rightarrow \mathbb{Z}_+,$ we say that the packing $\mathcal{B}$  is $(f, g)$-\textit{bounded} if for every vertex $v$ of $G,$ $v$ is a root in at least $f(v)$ and at most $g(v)$ rooted forests in $\mathcal{B}.$ For two non-negative integers $\alpha$ and $\beta$, we say that the packing $\mathcal{B}$  is $(\alpha, \beta)$-\textit{limited} if the total number of roots of the rooted forests in $\mathcal{B}$  is at least $\alpha$ and at most $\beta$. Given  a multiset $S$ of $V$ and a matroid {\sf M} on $S,$ a packing of rooted trees is called \textit{(complete) ${\sf M}$-based} if the multiset of roots of the rooted trees in the packing is (the set $S$) a subset of $S$ and the set of roots of the rooted trees containing $v$ in the packing is a basis of ${\sf M}$ for every vertex $v$ of $G$. For a family ${\cal S}$ of subsets of $V$ and  a matroid ${\sf M}$ on ${\cal S},$ a packing of  rooted forests in $G$ is called {\it ${\sf M}$-based} if there exists $S'\subseteq S$ for every $S\in {\cal S}$ such that $\{S':S\in {\cal S}\}$ is the set of the root sets of the rooted forests in the packing and for every vertex $v$ of $G,$ $\{S\in {\cal S}:$ rooted forests $(S',F)$ in the packing contains $v\}$ is a basis of ${\sf M}.$

Let $\bm{\mathcal{G}} = (V, \mathcal{E})$ be a hypergraph with vertex set $V$ and \textit{hyperedge set} $\bm{\mathcal{E}}$. A \textit{hyperedge} is a subset of $V$ of size at least two. For a partition $\mathcal{P}$ of $V$ and a hyperedge set $\mathcal{F}$, we denote by $\bm{e_{\mathcal{F}}(\mathcal{P})}$ the number of hyperedges in $\mathcal{F}$ that are not contained in a member of $\mathcal{P}$. If $\mathcal{F}$ is an edge set $F$, then $\bm{e_{{F}}(\mathcal{P})}$ is the number of edges in ${F}$ that are  between different members of $\mathcal{P}$. The operation that consists in replacing a hyperedge $Z$ by an edge whose end-vertices belong to $Z$ is called \textit{trimming}. The \textit{trimming} of a hypergraph consists in the trimming of all its hyperedges, resulting in a graph. The couple $(S, \mathcal{F})$ is a \textit{rooted hyperforest} if $\mathcal{F}$ can be trimmed to a forest $F$ such that for the graph $F'$ obtained from $F$ by deleting the isolated vertices not in $S,$ $(S, F')$ is a rooted forest. A set $\mathcal{B}$ of rooted hyperforests in $\mathcal{G}$ is called a \textit{packing} if $\mathcal{B}$ can be trimmed to a packing $\mathcal{B}'$ of rooted forests. Furthermore, $\mathcal{B}$ is said to be  ${\sf M}$-\textit{based} if  $\mathcal{B}'$ is ${\sf M}$-based.

\section{Results on partitions}\label{partitions}

In this section we present and demonstrate the necessary  results on functions on partitions. We hope that these results on supermodular functions on partitions  will have interesting applications later on as well. Actually, the results of this section will allow us to prove the new results on packing trees in Section \ref{packres}.

\subsection{Submodularity on partitions}

We here introduce two  submodular  and two supermodular functions on partitions of a set $V.$
\medskip

We start with the following observation about the uncrossing method on partitions which comes from \cite{HMSz}.

\begin{cl}[Hoppenot, Szigeti \cite{HMSz}]\label{lkbzlkbl1} 
For all partitions $\mathcal{P}_1$ and $\mathcal{P}_2$ of a set $V$ and  $X\subseteq V$, we have

(a) If $X$ crosses $\mathcal{P}_1 \sqcup \mathcal{P}_2$, then it crosses both $\mathcal{P}_1$ and $\mathcal{P}_2$.

(b) If $X$ crosses $\mathcal{P}_1 \sqcap \mathcal{P}_2$, then it crosses $\mathcal{P}_1$ or $\mathcal{P}_2$. 

(c) $|\mathcal{P}_1|+|\mathcal{P}_2|=|\mathcal{P}_1 \sqcap \mathcal{P}_2|+|\mathcal{P}_1 \sqcup \mathcal{P}_2|.$
\end{cl}

In  the proof of Theorem 9.5.1 in \cite{book}, Frank proved that  for a graph $G=(V,E)$, $e_E$ is supermodular on the partitions of $V.$ We generalized this in \cite{HS1} by showing that for a hypergraph $\mathcal{G}=(V,\mathcal{E})$, $e_\mathcal{E}$ is supermodular on the partitions of $V.$ Here we propose the following further extension that we will need to be able to introduce a new matroid in Subsection \ref{mbport}.

\begin{Lemma}\label{submodeap}
Let $\mathcal{G}=(V,\mathcal{E})$ be a hypergraph, $\mathcal{E}_1, \mathcal{E}_2\subseteq\mathcal{E}$ and $\mathcal{P}_1,\mathcal{P}_2$ partitions of $V.$ Then
\begin{equation}\label{submodeepeq}
				e_{\mathcal{E}_1}(\mathcal{P}_1) + 
				e_{\mathcal{E}_2}(\mathcal{P}_2) \ge 
				e_{\mathcal{E}_1\cap\mathcal{E}_2}(\mathcal{P}_1 \sqcap \mathcal{P}_2) + 
				e_{\mathcal{E}_1\cup\mathcal{E}_2}(\mathcal{P}_1 \sqcup \mathcal{P}_2).
			\end{equation}
\end{Lemma}

\begin{proof}
For any $X \in(\mathcal{E}_1\cup\mathcal{E}_2)\setminus (\mathcal{E}_1\cap\mathcal{E}_2)$, we have $e_{\{X\}}(\mathcal{P}_1 \sqcap \mathcal{P}_2)=0.$ If $X$ contributes to the right hand side of \eqref{submodeepeq}, then $e_{\{X\}}(\mathcal{P}_1 \sqcup \mathcal{P}_2) = 1$, hence, by Claim \ref{lkbzlkbl1}(a), $X$ contributes at least one to the left hand side of \eqref{submodeepeq} and exactly one to the right hand side of \eqref{submodeepeq}. For any $X \in\mathcal{E}_1\cap\mathcal{E}_2$, if $e_{\{X\}}(\mathcal{P}_1\sqcup\mathcal{P}_2)=1$, then, by Claim \ref{lkbzlkbl1}(a), $X$ contributes two to the left hand side of \eqref{submodeepeq} and at most two to the right hand side of \eqref{submodeepeq}. Suppose now that $e_{\{X\}}(\mathcal{P}_1 \sqcup \mathcal{P}_2) = 0$. If $X$ contributes to the right hand side of \eqref{submodeepeq}, then $e_{\{X\}}(\mathcal{P}_1\sqcap\mathcal{P}_2)=1$, hence, by Claim \ref{lkbzlkbl1}(b), $X$ contributes at least one to the left hand side of \eqref{submodeepeq} and exactly one to the right hand side of \eqref{submodeepeq}.  It follows that \eqref{submodeepeq} holds.
\end{proof}

If we are given an intersecting submodular function $b$ on a set $V$ and we define the value of a partition ${\cal P}$ of $V$ as the sum of the $b$-values of the members of ${\cal P}$, then we obtain a submodular function on the partitions of $V.$ We will need the following extension of this observation in the proof of the submodularity of the rank function of the above mentioned matroid.

\begin{Lemma}\label{submodrtx}
Let $S$ be a multiset of a set $V$ and $b$  an intersecting submodular function on $S.$ For all $S^1,S^2\subseteq S$ and $\mathcal{P}_1,\mathcal{P}_2$ partitions of $V,$ we have 
	\begin{equation}\label{submodrtxin}
	\sum_{X\in \mathcal{P}_1}b(S^1_X)+\sum_{Y\in \mathcal{P}_2}b(S^2_Y)	\ge	\sum_{Z\in \mathcal{P}_1\sqcap\mathcal{P}_2}	b((S^1\cap S^2)_Z)+\sum_{W\in \mathcal{P}_1\sqcup\mathcal{P}_2}b((S^1\cup S^2)_{W}).
	\end{equation}
\end{Lemma}

\begin{proof}
Let $S^1,S^2\subseteq S$ and $\mathcal{P}_1,\mathcal{P}_2$ be partitions of $V.$
Let $\bm{\mathcal{Q}}= \{(X, S^1_X) : X \in \mathcal{P}_1\} \cup \{(Y, S^2_Y) : Y \in \mathcal{P}_2\}$. For $U \subseteq V$ and $R \subseteq S_U$, let val$(U, R) = b(R)$. We define val$(\mathcal{Q}) = \sum_{(U,R) \in \mathcal{Q}}$val$ (U,R)=\sum_{X \in \mathcal{P}_1} b(S^1_X) + \sum_{Y \in \mathcal{P}_2} b(S^2_Y)$. 
We say that two couples $(U_1, R_1)$ and $(U_2,R_2)$ in $\mathcal{Q}$ are {\it properly intersecting} if $U_1$ and $U_2$ are properly intersecting sets. The  uncrossing of such two couples is the operation that replaces them by $(U_1 \cap U_2, R_1 \cap R_2)$ and $(U_1 \cup U_2, R_1 \cup R_2)$.  
We apply the uncrossing operation on $\mathcal{Q}$ to obtain a new family $\mathcal{Q}''$ which contains no properly intersecting couples. 
Actually, the uncrossing of $\mathcal{Q}$ will mimic the uncrossing of $\mathcal{P}_1$ and $\mathcal{P}_2$.
Note that in each step $Z = U_1 \cap U_2$ will be a member of $\mathcal{P}_1 \sqcap \mathcal{P}_2$ and that $R_1 \cap R_2 = S^1_Z \cap S^2_Z$.
Note also that when $W = U_1 \cup U_2$ becomes a member of $\mathcal{P}_1 \sqcup \mathcal{P}_2$, then $R_1 \cup R_2 = S^1_W \cup S^2_W$.
It follows that the value of  $\mathcal{Q}''$ is $\sum_{Z \in \mathcal{P}_1 \sqcap \mathcal{P}_2} b((S^1 \cap S^2)_Z) + \sum_{W \in \mathcal{P}_1 \sqcup \mathcal{P}_2} b((S^1 \cup S^2)_{W})$. If $\mathcal{Q}_{i+1}$ is obtained from $\mathcal{Q}_i$ by uncrossing two couples, then, by the intersecting submodularity of $b,$ we have val$(\mathcal{Q}_i)\ge $ val$(\mathcal{Q}_{i+1})$. Hence the lemma follows.
\end{proof}

We introduce two  supermodular functions on partitions that we will need later.

\begin{cl}\label{icvhjbkn}
Let $S$ be a multiset of a set $V$, $\beta\in\mathbb Z_+,$ $f, g: V\rightarrow \mathbb Z_+$ functions, and ${\sf M}=(S,r_{\sf M})$ a matroid. Let the functions $p_1$ and $p_2$ be defined as follows. For every partition $\mathcal{P}$ of $V,$ 
\begin{eqnarray}
	\text{{\boldmath$p_1$}}(\mathcal{P}) &=&-g(V)+\sum_{X\in {\cal P}} \max\{r_{{\sf M}}(S)+g(Y)-r_{{\sf M}}(S_Y):Y\subseteq X\},	\label{defpg}\\
	\text{{\boldmath$p_2$}}(\mathcal{P}) &=&-\beta+\sum_{X\in {\cal P}} \max\{r_{{\sf M}}(S)+f(Y)-r_{{\sf M}}(S_Y):Y\subseteq X\}.	\label{defpf}
\end{eqnarray}
The functions $p_1$ and $p_2$ are supermodular on partitions of $V.$
\end{cl}

\begin{proof}
Since $r_{{\sf M}}(S)$ is constant, $g$ and $f$ are modular and $r_{{\sf M}}$ is submodular, $r_{{\sf M}}(S)+g-r_{{\sf M}}$ and $r_{{\sf M}}(S)+f-r_{{\sf M}}$ are supermodular. Then, by \cite[Theorem 14.3.1]{book}, $\max\{r_{{\sf M}}(S)+g(Y)-r_{{\sf M}}(S_Y):Y\subseteq X\}$ and $\max\{r_{{\sf M}}(S)+f(Y)-r_{{\sf M}}(S_Y):Y\subseteq X\}$ are supermodular. It follows that $\sum_{X\in {\cal P}} \max\{r_{{\sf M}}(S)+g(Y)-r_{{\sf M}}(S_Y):Y\subseteq X\}$ and $\sum_{X\in {\cal P}} \max\{r_{{\sf M}}(S)+f(Y)-r_{{\sf M}}(S_Y):Y\subseteq X\}$ are supermodular on partitions of $V.$ Then, since $g(V)$ and $\beta$ are constant,  we may conclude that the functions $p_1$ and $p_2$ are supermodular on partitions of $V.$
\end{proof}

\subsection{Covering two supermodular functions on the partitions}

In edge-connectivity augmentation problems the aim is to cover a function on the subsets of vertices by a set of edges. 
The directed version was considered in Corollary 2.48 of \cite{FAKT08}.
Here we have to cover a function on partitions of a vertex set by an edge set. 
We can even cover two such functions simultaneously. 

\begin{thm}\label{coveringpartsuper12}
Let $p_1$ and $p_2$ be supermodular functions on the partitions of a set $V$  and $\gamma\in \mathbb{Z}_+.$ There exists an edge set $F$ on $V$ of size $\gamma$ such that 
\begin{eqnarray}\label{partcov}
	e_F(\mathcal{P})	&	\ge	&	\max\{p_1(\mathcal{P}), p_2(\mathcal{P})\} \hskip 1truecm \text{ for every partition } \mathcal{P}  \text{ of } V\hskip .8truecm
\end{eqnarray}
if and only if 
\begin{eqnarray}\label{partcovcond}
	0	&	\ge	&	\max\{p_1(\{V\}),p_2(\{V\})\},\label{partcovcond1}\\
	\gamma	&	\ge	&	\max\{p_1(\mathcal{P}), p_2(\mathcal{P})\} \hskip 1truecm \text{ for every partition } \mathcal{P}  \text{ of } V.\label{partcovcond2}
\end{eqnarray}
\end{thm}

\begin{proof}
Since the necessity is immediate, we only prove the sufficiency. It is enough to prove the theorem for {\boldmath$\gamma$} $=\max\{p_1(\mathcal{P}), p_2(\mathcal{P}):\mathcal{P}\text{ partition of } V\}.$ The proof is by induction on $\gamma.$ If $\gamma=0$, then there is nothing to prove. Suppose that the theorem is true for $\gamma-1\ge 0.$ Let {\boldmath$p_1$} and {\boldmath$p_2$} be two supermodular functions on the partitions of a set $V$ such that $p_1, p_2$ and  $\gamma$ satisfy \eqref{partcovcond1} and \eqref{partcovcond2}. Let {\boldmath$\mathcal{Q}_1$} $:=\{\mathcal{P}$ partition of $V: p_1(\mathcal{P})=\gamma\}$ and {\boldmath$\mathcal{Q}_2$} $:=\{\mathcal{P}$ partition of $V: p_2(\mathcal{P})=\gamma\}.$ Note that at least one of $Q_1$ and $Q_2$ is not empty.
\begin{cl}\label{intunipart}
	If $\mathcal{P}_1,\mathcal{P}_2\in\mathcal{Q}_i,$ then $\mathcal{P}_1\sqcap\mathcal{P}_2,\mathcal{P}_1\sqcup\mathcal{P}_2\in\mathcal{Q}_i$ for $i=1,2.$
\end{cl}
\begin{proof}
By $\mathcal{P}_1,\mathcal{P}_2\in\mathcal{Q}_i$, $p_i$ is supermodular on the partitions of $V,$ and \eqref{partcovcond2}, we have $\gamma+\gamma=p_i(\mathcal{P}_1)+p_i(\mathcal{P}_2)\le p_i(\mathcal{P}_1\sqcap\mathcal{P}_2)+p_i(\mathcal{P}_1\sqcup\mathcal{P}_2)\le\gamma+\gamma,$ so equality holds everywhere and the claim follows.
\end{proof}
If $\mathcal{Q}_i\neq\emptyset,$ then let {\boldmath$X_i$} be a maximal set among the members of the partitions in $\mathcal{Q}_i$ and {\boldmath$\mathcal{P}_i^1$}   $\in\mathcal{Q}_i$ such that  $X_i\in\mathcal{P}_i^1.$ Since, by $\mathcal{P}_i^1\in\mathcal{Q}_i$, $\gamma>0,$ and \eqref{partcovcond1}, we have $p_i(\mathcal{P}_i^1)=\gamma>0\ge p_i(\{V\}),$ we get that $\emptyset\neq X_i\neq V.$ Thus there exists $u_i\in X_i$ and $v_i\in \overline X_i.$ 

Therefore, if $\mathcal{Q}_1\neq\emptyset\neq \mathcal{Q}_2$, then there exists $u,v\in V$ such that both $X_1$ and $X_2$ separate $u$ and $v.$ If only one of $\mathcal{Q}_1$ and $\mathcal{Q}_2$ is non-empty, say $\mathcal{Q}_i$, then let $u=u_i$ and $v=v_i.$
\begin{cl}\label{euvpmin}
		$e_{uv}(\mathcal{P})=1$ for every  $\mathcal{P}\in\mathcal{Q}_1\cup\mathcal{Q}_2.$
\end{cl}
\begin{proof}
Suppose that there exists a partition in $\mathcal{Q}_1\cup\mathcal{Q}_2$, say $\bm{\mathcal{P}_i^2}\in\mathcal{Q}_i$ such that $e_{uv}(\mathcal{P}_i^2)=0$, that is a set {\boldmath$Y_i$} $\in \mathcal{P}_i^2$ contains both $u$ and $v.$ Recall that $X_i$ contains exactly one of $u$ and $v$, say $u$. By Claim \ref{intunipart}, we have $\mathcal{P}_i^1\sqcup\mathcal{P}_i^2\in\mathcal{Q}_i.$ By Claim \ref{lkbzlkbl1}, we get that  an element of  $\mathcal{P}_i^1 \sqcup \mathcal{P}_i^2$ contains $X_i$ and an element of  $\mathcal{P}_i^1 \sqcup \mathcal{P}_i^2$ contains $Y_i$. Since $u\in X_i\cap Y_i$, it follows that an element {\boldmath$Z_i$}  of  $\mathcal{P}_i^1 \sqcup \mathcal{P}_i^2$ contains $X_i\cup Y_i.$ The fact that $X_i\subset X_i\cup\{v\}\subseteq X_i\cup Y_i\subseteq Z_i$ contradicts the maximality of $X_i.$
\end{proof}
Let {\boldmath$p_i'$}$(\mathcal{P})=p_i(\mathcal{P})-e_{uv}(\mathcal{P})$ for every partition $\mathcal{P}$ of $V$ and for $i=1,2.$ Since, by assumption and \eqref{submodeap}, $p_i$ and $-e_{uv}$ are supermodular  on the partitions of $V$, so is $p_i'.$ Note that, by \eqref{partcovcond1}, we have $\max\{p'_1(\{V\}),p'_2(\{V\})\}=\max\{p_1(\{V\}),p_2(\{V\})\}\le 0.$ Let {\boldmath$\gamma'$} $=\max\{p_1'(\mathcal{P}), p_2'(\mathcal{P}):\mathcal{P}\text{  partition of } V\}.$ By Claim \ref{euvpmin}, we have $\gamma'=\gamma-1.$ Then, by induction, there exists an edge set {\boldmath$F'$} on $V$ of size $\gamma'$ such that 	$e_{F'}(\mathcal{P})	\ge	\max\{p_1'(\mathcal{P}), p_2'(\mathcal{P})\}$ for every partition $\mathcal{P}$  of $V.$ Let {\boldmath$F$} $=F'\cup \{uv\}.$ Then $|F|=|F'|+1=\gamma'+1=\gamma$ and $e_{F}(\mathcal{P})=e_{F'}(\mathcal{P})+e_{uv}(\mathcal{P})\ge p_i'(\mathcal{P})+e_{uv}(\mathcal{P})=p_i(\mathcal{P})$ for every partition $\mathcal{P}$  of $V$ and for $i=1,2.$ Hence $F$ is the desired edge set for $p_1$ and $p_2.$
\end{proof}

\subsection{Trimming}

We proved in \cite{HMSz} that a hypergraph covering two particular supermodular functions on partitions of  a set can be trimmed to a graph covering  the same functions. Here we provide the general form of it. We will use this  to extend our result on graphs to hypergraphs. We hope there will be other applications of Theorem \ref{trimminglemma} later on.

\begin{thm}\label{trimminglemma}
Let $p_1$ and $p_2$ be supermodular functions on the partitions of $V$ and $\mathcal{G}=(V,\mathcal{E})$  a hypergraph. Then $\mathcal{G}$ can  be trimmed to a graph $G=(V,E)$  such that 
	\begin{eqnarray}
		e_{E}(\mathcal{P})	&	\ge	&	\max\{p_1(\mathcal{P}), p_2(\mathcal{P})\} 
		\hskip .4truecm \text{ for every partition } \mathcal{P}  \text{ of } V\label{trimmingcondgr}
	\end{eqnarray}
if and only if
	\begin{eqnarray}
		e_\mathcal{E}(\mathcal{P})	&	\ge	&	\max\{p_1(\mathcal{P}), p_2(\mathcal{P})\} 
		\hskip .4truecm \text{ for every partition } \mathcal{P}  \text{ of } V.\label{trimmingcond}
	\end{eqnarray}
\end{thm}

\begin{proof}
Since the necessity is immediate, we only prove the sufficiency. 
We prove the theorem by induction on $\sum_{X \in \mathcal{E}} |X|.$
If for every  $X \in \mathcal{E}$, $|X| = 2$, then $\mathcal{G}$ is a graph and, \eqref{trimmingcond}   coincides with \eqref{trimmingcondgr}. Otherwise, there exists a hyperedge {\boldmath$X$} $\in \mathcal{E}$ of size at least $3$. We show that we can  remove a vertex from $X$ without violating \eqref{trimmingcond}; and then the induction hypothesis completes the proof. Suppose for a contradiction that for every $v \in X$, \eqref{trimmingcond}  is violated after the removal of $v$ from $X$. By $|X| \ge 3$, there exist at least two vertices of $X$, say $v_1$ and $v_2$, such that  the removal $v_1$ and the removal of $v_2$ violate the same $p_i$. We fix this index $i$ for the rest of the proof. Since this condition is satisfied before the removal of the vertex, there exist partitions $\mathcal{P}_1$ and $\mathcal{P}_2$ of $V$, such that $p_i(\mathcal{P}_1)=e_\mathcal{E}(\mathcal{P}_1)$ and $p_i(\mathcal{P}_2)=e_\mathcal{E}(\mathcal{P}_2)$, and $e_{\mathcal{E}}(\mathcal{P}_j)$ decreases when we remove $v_j$ from $X$  for $j=1, 2$. It follows that $X\setminus \{v_j\}$ is contained in a member {\boldmath$Y_j$} of $\mathcal{P}_j$  for $j=1,2$; and hence, by $|X| \ge 3$, we have $Y_1\cap Y_2\supseteq X\setminus\{v_1,v_2\}\neq\emptyset.$ By $p_i(\mathcal{P}_1)=e_\mathcal{E}(\mathcal{P}_1)$ and $p_i(\mathcal{P}_2)=e_\mathcal{E}(\mathcal{P}_2)$, Claim \ref{submodeap}, \eqref{trimmingcond}, and $p_i$ is supermodular  on the partitions of $V$, we obtain that 
\begin{eqnarray*}
p_i(\mathcal{P}_1)+p_i(\mathcal{P}_2)	&	=	&	e_\mathcal{E}(\mathcal{P}_1)+e_\mathcal{E}(\mathcal{P}_2)
	\ \ =\ \ e_{\mathcal{E}-X}(\mathcal{P}_1)+e_{\mathcal{E}-X}(\mathcal{P}_2)\ \ +\ \ e_X(\mathcal{P}_1)+e_X(\mathcal{P}_2)\\
	&	\ge	&	 e_{\mathcal{E}-X}(\mathcal{P}_1\sqcap\mathcal{P}_2)+e_{\mathcal{E}-X}(\mathcal{P}_1\sqcup\mathcal{P}_2)\ \ +\ \ e_{X}(\mathcal{P}_1\sqcap\mathcal{P}_2)+e_{X}(\mathcal{P}_1\sqcup\mathcal{P}_2)\\
	&	=	&	e_{\mathcal{E}}(\mathcal{P}_1\sqcap\mathcal{P}_2)+e_{\mathcal{E}}(\mathcal{P}_1\sqcup\mathcal{P}_2)
	\ \	\ge	\ \	 p_i(\mathcal{P}_1\sqcap\mathcal{P}_2)+p_i(\mathcal{P}_1\sqcup\mathcal{P}_2)\\
	&	\ge	&	 p_i(\mathcal{P}_1)+p_i(\mathcal{P}_2).
\end{eqnarray*}
We hence have equality everywhere, in particular, $e_X({\cal P}_1)+e_X({\cal P}_2)=e_X(\mathcal{P}_1\sqcap\mathcal{P}_2)+e_X(\mathcal{P}_1\sqcup\mathcal{P}_2).$ Thus, since $X$ crosses both ${\cal P}_1$ and ${\cal P}_2$, we get that $X$ also crosses $\mathcal{P}_1\sqcup\mathcal{P}_2.$ However, by Claim \ref{lkbzlkbl1}, we get that a member of $\mathcal{P}_1\sqcup\mathcal{P}_2$ contains $Y_1$ and a member of $\mathcal{P}_1\sqcup\mathcal{P}_2$ contains $Y_2$. Since $Y_1\cap Y_2\neq\emptyset,$ it follows that a member of $\mathcal{P}_1\sqcup\mathcal{P}_2$ contains $Y_1\cup Y_2\supseteq X,$ which contradicts the fact that $X$  crosses $\mathcal{P}_1\sqcup\mathcal{P}_2.$
\end{proof}

\section{Results on packings}\label{packres}

In the previous section we proved all the necessary tools to be applied in this section. We now may present the results on packing trees and we are ready to prove them. This section contains seven subsections containing more and more general results, starting with the basic result on packing spanning trees, and finishing with a result on augmentation for matroid-based $(f,g)$-bounded $(\alpha,\beta)$-limited packing of  rooted  hyperforests. 

\subsection{Packing of spanning trees}

The classic result on packing spanning trees is due to Nash-Williams~\cite{NW} and Tutte~\cite{Tu}. We provide a new proof of it that imitates the proof of Lov\'asz, which he gave in \cite{Lov} for Edmonds' theorem on packing spanning arborescences, and is inspired by Theorem 10.4.4 in \cite{book}. This method of Lov\'asz has been successfully applied   in more general settings as well. We hope that our method will also be applied later. We think that it is natural that this  method works for the undirected case as well and that this fact is worth being  known.

\begin{thm} [Nash-Williams~\cite{NW}, Tutte~\cite{Tu}]\label{tuttethmrem} 
Let $G=(V,E)$ be a graph and $k\in \mathbb{Z}_+.$
There exists a packing of $k$ spanning trees in $G$  if and only if  
	\begin{eqnarray} \label{tuttecond} 
		e_E({\cal P}) &\geq& k(|{\cal P}|-1)\hskip .5truecm  \text{ for every partition  ${\cal P}$ of  $V.$}
	\end{eqnarray}  
\end{thm}

\begin{proof}
We prove only the difficult direction. Let $(G=(V,E),k)$ be a smallest counter-example. Then $k\ge 1.$ The following claim is well-known, see for example in \cite{book}.

\begin{cl}\label{tightpart}
$|E|=k(|V|-1).$
\end{cl}

\begin{proof}
By \eqref{tuttecond} applied for $\{v\}_{v\in V}$, we obtain that $|E|\ge k(|V|-1).$ Suppose for a contradiction that $|E|>k(|V|-1).$ Hence \eqref{tuttecond} is strict for the partition $\{v\}_{v\in V}.$ 

If there is no partition of $V$, other than $\{V\}$, that satisfies \eqref{tuttecond} with equality, then we can delete any edge of $G$ to obtain $G'$ that also satisfies \eqref{tuttecond}. Hence, by the minimality of $G$, there exists a packing of $k$ spanning trees in $G'$, and hence in $G,$ which is a contradiction. 

So there is a partition $\bm{{\cal P}}\neq \{V\}$ of $V$ that satisfies \eqref{tuttecond} with equality. Note that ${\cal P}\neq \{v\}_{v\in V}.$ Then there exists a member $\bm{X}$ $\in {\cal P}$ such that $1<|X|<|V|.$

We show that both $G[X]$ and $G/X$ satisfy  \eqref{tuttecond}. First, if a partition ${\cal P}'$ of $X$ violated \eqref{tuttecond} in $G[X]$, then ${\cal P}''=({\cal P}-\{X\})\cup{\cal P}'$ would violate \eqref{tuttecond} in $G.$ Indeed, $e_E({\cal P}'')=e_E({\cal P})+e_E({\cal P}')<k(|{\cal P}|-1)+k(|{\cal P}'|-1)=k(|{\cal P}''|-1).$ Second, if a partition ${\cal P}'$ of $V(G/X)$ violated \eqref{tuttecond} in $G/X$, then ${\cal P}''=({\cal P}'\setminus\{Y\})\cup\{(Y\setminus \{v_X\})\cup X\}$ would violate \eqref{tuttecond} in $G,$ where the contracted vertex $v_X$ is in $Y\in {\cal P}'.$  Indeed, $e_E({\cal P}'')=e_E({\cal P}')<k(|{\cal P}'|-1)=k(|{\cal P}''|-1).$ 

Since $G[X]$ and $G/X$ satisfy  \eqref{tuttecond} and $(G,k)$ is a smallest counter-example, there exist a packing $T_1,\dots, T_k$ of $k$ spanning trees in $G[X]$ and a packing $T'_1,\dots, T'_k$ of $k$ spanning trees in $G/X.$ By replacing the vertex $v_X$ in each $T'_i$ by $T_i$, we obtain a packing  of $k$ spanning trees in $G,$ which is a contradiction.
\end{proof}
For any $\emptyset\neq X\subseteq V$ and ${\cal P}_X=\{X\}\cup\{v\}_{v\in V\setminus X},$ by  $|E|=k(|V|-1)$ and \eqref{tuttecond}, we get 
	\begin{eqnarray} \label{covercond}
		i_E(X)-k(|X|-1)=k|V\setminus X|-e_E({\cal P}_X)=k(|{\cal P}_X|-1)-e_E({\cal P}_X)\le 0.
	\end{eqnarray}  

Let $\bm{T}=(S,F)$ be a maximal tree in $G$ satisfying \eqref{extensioncond}. For $s\in V$, by \eqref{covercond}, $T_s=(s,\emptyset)$ satisfies \eqref{extensioncond}, so $T$ exists. 
	\begin{eqnarray} \label{extensioncond}
		i_{E\setminus F}(X)	&	\le	&	 k(|X|-1)-|X\cap S|+1=:\bm{m(X)}	\qquad\text{for every }\emptyset\neq X\subseteq V.	
	\end{eqnarray}  

\begin{Lemma}\label{spantree}
$T$ is a spanning tree of $G.$ 
\end{Lemma}

\begin{proof}
Suppose that $S\neq V.$ A vertex set $X$ is called {\it tight} if \eqref{extensioncond} holds with equality. By $\eqref{covercond}$, every tight set intersects $S.$ A tight set $X$ is {\it dangerous} if $X\setminus S\neq \emptyset.$ Note that, by  $|E|=k(|V|-1)$ and $S\neq V$, $V$ is dangerous. Thus, there exists a minimal dangerous set $\bm{X}.$ Then, by \eqref{extensioncond} and \eqref{covercond}, we have $d_{E\setminus F}(X\cap S,X\setminus S)=i_{E\setminus F}(X)-i_{E\setminus F}(X\cap S)-i_{E\setminus F}(X\setminus S)\ge k(|X|-1)-|X\cap S|+1-(k-1)(|X\cap S|-1)-k(|X\setminus S|-1)=k\ge 1,$ so there exists an edge $uv$ from $u\in X\cap S$ to $v\in X\setminus S.$ Note that $T'=(S\cup \{v\}, F\cup \{uv\})$ is a tree. By the maximality of $T$, we get that there exists a set $Y$ such that $i_{E\setminus (F\cup \{uv\})}(X)>k(|X|-1)-|X\cap (S\cup \{v\})|+1.$ Then, by \eqref{extensioncond}, 
we get that $v\in Y$ and $u\in V\setminus Y.$ Observe that $v\in (X\cap Y)\setminus S.$ Then, by the tightness of $X$ and $Y$, the modularity of $m$, $X\cap Y\neq\emptyset,$ \eqref{extensioncond}, and supermodularity of $i_{E\setminus F},$ we have $i_{E\setminus F}(X)+i_{E\setminus F}(Y)=m(X)+m(Y)=m(X\cap Y)+m(X\cup Y)\ge i_{E\setminus F}(X\cap Y)+i_{E\setminus F}(X\cup Y)\ge i_{E\setminus F}(X)+i_{E\setminus F}(Y),$ so equality holds everywhere. In particular, $X\cap Y$ is dangerous. Since $u\in X\setminus Y,$ this contradicts the minimality of $X,$ and the proof of the lemma is completed. 
\end{proof}
By Lemma \ref{spantree}, $T$ is a spanning tree of $G,$ so $S=V$ and $|F|=|V|-1.$ Then, by  $|E|=k(|V|-1)$ and \eqref{extensioncond}, for every partition ${\cal P}$ of $V$, we have $(k-1)(|V|-1)-e_{E\setminus F}({\cal P})=\sum_{X\in{\cal P}}i_{E\setminus F}(X)\le \sum_{X\in{\cal P}}(k-1)(|X|-1)=(k-1)(|V|-1)-(k-1)(|{\cal P}|-1)$, that is $(G-F,k-1)$ satisfies the condition \eqref{tuttecond}. Hence, by the minimality of $(G,k)$, there exists a packing of $k-1$ spanning trees in $G-F$. By adding the spanning tree $T$ of $G,$ we get a packing of $k$ spanning trees in $G$ which is a contradiction.
\end{proof}

\subsection{Matroid-based packing of rooted trees}\label{mbport}

A nice extension of Theorem \ref{tuttethmrem} with some matroid constraint was proposed in \cite{KT}.

\begin{thm}[Katoh, Tanigawa \cite{KT}] \label{thmKT}
Let $G=(V,E)$ be a graph, $S$ a multiset of vertices in $V$ and ${\sf M}=(S,\mathcal{I}_{\sf M})$ a matroid with rank function $r_{\sf M}$. There exists a complete ${\sf M}$-based packing of rooted trees in $G$  if and only if  
\begin{eqnarray}
	S_v		&	\in	& 	\mathcal{I}_{\sf M} 	\hskip 3.45truecm	\text{ for every } v\in V, 	\label{matcondori1}	\\
e_E({\cal P})	&	\geq	& 	\sum_{X\in {\cal P}} (r_{{\sf M}}(S)-r_{{\sf M}}(S_X)) 	\hskip .44truecm	\text{ for every partition } {\cal P} \text{ of }  V.\label{matcondKT2}
\end{eqnarray}
\end{thm}

If $S$ is a multiset of vertices in $V$ of size $k$ and ${\sf M}$ is the free matroid on $S$, then Theorem \ref{thmKT} reduces to Theorem \ref{tuttethmrem}.
\medskip

Theorem \ref{thmKT} was deduced from the following result in \cite{KT}.

\begin{thm}[Katoh, Tanigawa \cite{KT}] \label{thmKTimplicit}
Let $G=(V,E)$ be a graph, $S$ a multiset of vertices in $V$ and ${\sf M}=(S,\mathcal{I}_{\sf M})$ a matroid with rank function $r_{\sf M}$ that satisfies \eqref{matcondori1}. Let the function $r_{KT}$ be defined as follows, for every $F\subseteq E,$
\begin{eqnarray}\label{KTfunction}
r_{KT}(F)=r_{{\sf M}}(S)|V|-|S|+\min\{e_F(\mathcal{P})-\sum_{X\in {\cal P}} (r_{{\sf M}}(S)-r_{{\sf M}}(S_X)):\mathcal{P} \text{ partition of } V\}.
\end{eqnarray}
(a) Then $r_{KT}$ is the rank function of a matroid ${\sf M}_{KT}$.

\noindent (b)  $F\subseteq E$ is the edge set of  a complete ${\sf M}$-based packing of rooted trees in $G$  if and only if 
\begin{eqnarray}
F \text{ is independent in } {\sf M}_{KT},\label{uovyiytutr}\\
|F|=r_{{\sf M}}(S)|V|-|S|.\label{uovyiytutr2}
\end{eqnarray}
\end{thm}

We propose the following more general result to be applied later.

\begin{thm}\label{thmKTimplicitgen}
Let $G=(V,E)$ be a graph, $S$ a multiset of vertices in $V$, and ${\sf M}=(S,\mathcal{I}_{\sf M})$ a matroid with rank function $r_{\sf M}$. Let the function $\bm{r'_{KT}}$ be defined as follows, for every $F\subseteq E, T\subseteq S,$
\begin{eqnarray}\label{ZHfunction}
	r'_{KT}(F\cup T)=r_{{\sf M}}(S)|V|+\min\{e_F(\mathcal{P})-\sum_{X\in {\cal P}} (r_{{\sf M}}(S)-r_{{\sf M}}(T_X)):\mathcal{P} \text{ partition of } V\}.
\end{eqnarray}
(a) Then $r'_{KT}$ is the rank function of a matroid $\bm{{\sf M}'_{KT}}$.

\noindent (b)  $F\subseteq E$ is the edge set and $T\subseteq S$ is the root set of  an ${\sf M}$-based packing of rooted trees in $G$  if and only if  
\begin{eqnarray}
&&F\cup T \text{ is independent in } {\sf M}'_{KT},\label{tfiuyfuog}\\
&&|F\cup T|=r_{{\sf M}}(S)|V|.\label{tfiuyfuog2}
\end{eqnarray}
\end{thm}

\begin{proof}
(a) It is clear that $r'_{KT}$ is integer-valued. By \eqref{ZHfunction}, $e_\cdot(\cdot)\ge 0,$ and $r_{{\sf M}}(T_\cdot)\ge 0,$  we have $r'_{KT}(F\cup T)=\min\{e_F(\mathcal{P})+\sum_{X\in {\cal P}} ((|X|-1)r_{{\sf M}}(S)+r_{{\sf M}}(T_X)):\mathcal{P} \text{ partition of } V\}\ge 0$ for every $F\subseteq E, T\subseteq S.$ Since the functions $e_\cdot$ and $r_{{\sf M}}(T_\cdot)$ are non-decreasing, so is $r'_{KT}$. 
For every $F\subseteq E, T\subseteq S,$ by taking the partition $\{v\}_{v\in V}$ in \eqref{ZHfunction}, and by the subcardinality of $r_{{\sf M}},$ we have 
\begin{eqnarray*}
r'_{KT}(F\cup T)\le r_{{\sf M}}(S)|V|+|F|-r_{{\sf M}}(S)|V|+\sum_{v\in V}r_{{\sf M}}(T_v)\le |F|+\sum_{v\in V}|T_v|=|F\cup T|, 
\end{eqnarray*}
so $r'_{KT}$ is subcardinal. 

To show the submodularity of $r'_{KT}$, let $F_1, F_2\subseteq E, T^1,T^2\subseteq S,$ and $\mathcal{P}_1, \mathcal{P}_2$ the partitions that provide $r'_{KT}(F_1\cup T^1)$ and $r'_{KT}(F_2\cup T^2)$. Then, by Lemma \ref{submodeap} applied for $F_1$ and $F_2$ and Lemma \ref{submodrtx} applied for $b(\cdot)=r_{{\sf M}}(S_\cdot)-r_{{\sf M}}(S)$,  we have 
\begin{eqnarray*}
	r'_{KT}(F_1\cup T^1)+r'_{KT}(F_2\cup T^2)
	&	=	&	r_{{\sf M}}(S)|V|+e_{F_1}(\mathcal{P}_1)+\sum_{X\in {\cal P}_1}(r_{{\sf M}}(T^1_X)-r_{{\sf M}}(S))\\
	&	+	&	r_{{\sf M}}(S)|V|+e_{F_2}(\mathcal{P}_2)+\sum_{Y\in {\cal P}_2}(r_{{\sf M}}(T^2_Y)- r_{{\sf M}}(S))\\
	&	\ge	&	r_{{\sf M}}(S)|V|+e_{F_1\cap F_2}(\mathcal{P}_1\sqcap\mathcal{P}_2)+
				\sum_{Z\in \mathcal{P}_1\sqcap\mathcal{P}_2}(r_{{\sf M}}((T^1\cap T^2)_Z)-r_{{\sf M}}(S))\\
	&	+	&	r_{{\sf M}}(S)|V|+e_{F_1\cup F_2}(\mathcal{P}_1\sqcup\mathcal{P}_2)+
				\sum_{W\in \mathcal{P}_1\sqcup\mathcal{P}_2} (r_{{\sf M}}((T^1\cup T^2)_W)-r_{{\sf M}}(S))\\
	&	\ge	&	r'_{KT}((F_1\cap F_2)\cup (T^1\cap T^2))+r'_{KT}((F_1\cup F_2)\cup (T^1\cup T^2))\\
	&	=	&	r'_{KT}((F_1\cup T^1)\cap (F_2\cup T^2))+r'_{KT}((F_1\cup T^1)\cup (F_2\cup T^2)).
\end{eqnarray*}
Thus $r'_{KT}$ is submodular. From the previous arguments, it follows that $r'_{KT}$ is the rank function of a matroid ${\sf M}'_{KT}.$
\medskip

(b) To prove the {\bf necessity}, suppose that $F\subseteq E$ is the edge set and $T\subseteq S$ is the root set of  an ${\sf M}$-based packing of rooted trees in $G.$ Then, $T_v\in\mathcal{I}_{\sf M}$  for every $v\in V.$ So, by Theorem \ref{thmKTimplicit} applied for $T$, $r_{{\sf M}}(S)|V|=|F|+|T|=r_{KT}(F)+|T|=r_{{\sf M}}(S)|V|+\min\{e_F(\mathcal{P})-\sum_{X\in {\cal P}} (r_{{\sf M}}(S)-r_{{\sf M}}(T_X)):\mathcal{P} \text{ partition of } V\}=r'_{KT}(F\cup T),$ hence \eqref{tfiuyfuog} and \eqref{tfiuyfuog2} hold.
\medskip

To prove the {\bf sufficiency}, suppose that \eqref{tfiuyfuog} and \eqref{tfiuyfuog2} hold for some $F\subseteq E$ and $T\subseteq S$. Then, by taking the partition   $\mathcal{P}=\{V\}$ in \eqref{ZHfunction} and by the monotonicity of $r_{{\sf M}},$ we get that  $0=|F\cup T|-r_{{\sf M}}(S)|V|=r'_{KT}(F\cup T)-r_{{\sf M}}(S)|V|\le r_{{\sf M}}(T)-r_{{\sf M}}(S)\le 0,$ so $r_{{\sf M}}(T)=r_{{\sf M}}(S).$ Note also that, by \eqref{tfiuyfuog}, $T$ is independent in ${\sf M}'_{KT}$. Then, by taking the partition $\{v\}_{v\in V}$ in \eqref{ZHfunction} and by the subcardinality of $r_{{\sf M}},$ we have $|T|=r'_{KT}(T)\le r_{{\sf M}}(S)|V|-r_{{\sf M}}(S)|V|+\sum_{v\in V}r_{{\sf M}}(T_v)\le \sum_{v\in V}|T_v|=|T|$, so equivality holds everywhere, that is $T_v\in\mathcal{I}_{{\sf M}}$ for every $v\in V.$ Let ${\sf M}|_T$ be the matroid obtained from ${\sf M}$ by restricting it on $T.$ Then $T_v\in\mathcal{I}_{{\sf M}|_T}$ for every $v\in V.$ Let ${\sf M}^T_{KT}$ be the matroid of Theorem \ref{thmKTimplicit} with ground set  $T.$ Then, for its rank function, we have $r^T_{KT}(F)=r_{{\sf M}}(T)|V|-|T|+\min\{e_F(\mathcal{P})-\sum_{X\in {\cal P}} (r_{{\sf M}}(T)-r_{{\sf M}}(T_X)):\mathcal{P} \text{ partition of } V\}=r'_{KT}(F\cup T)-|T|=|F\cup T|-|T|=|F|,$ so $F$ is independent in ${\sf M}^T_{KT}$. Then, by Theorem \ref{thmKTimplicit}, there exists  a complete ${\sf M}|_T$-based packing of rooted trees in $G$ whose edge-set is $F,$ and we are done.
\end{proof}

Let us clarify the relation between the matroids ${\sf M}_{KT}$ and ${\sf M}'_{KT}.$ 

\begin{cl}\label{kiytckhjfvlbnk}
If \eqref{matcondori1} holds, then $S$ is independent in ${\sf M}'_{KT}$ and ${\sf M}'_{KT}/S={\sf M}_{KT}.$
\end{cl}

\begin{proof}
For the partition ${\cal P}$ that provides $r'_{KT}(S)$  in \eqref{ZHfunction}, by the monotonicity of $r_{{\sf M}}$ and \eqref{matcondori1}, we have
\begin{eqnarray*}
|S|&\ge& r'_{KT}(S)\ \ =\ \ r_{{\sf M}}(S)|V|-\sum_{X\in {\cal P}} (r_{{\sf M}}(S)-r_{{\sf M}}(S_X))\\
&=&\sum_{X\in {\cal P}}(r_{{\sf M}}(S_X)+(|X|-1)r_{{\sf M}}(S))\ \ \ge\ \ \sum_{X\in {\cal P}}\sum_{v\in X}r_{{\sf M}}(S_v)\ \ =\ \ \sum_{X\in {\cal P}}\sum_{v\in X}|S_v|\ \ =\ \ |S|.
\end{eqnarray*}
Further, by \eqref{KTfunction} and \eqref{ZHfunction}, we have $r_{KT}(F)=r'_{KT}(F\cup S)-|S|=r'_{KT}(F\cup S)-r'_{KT}(S)=(r'_{KT})_{/S}(F)$ for every $F\subseteq E.$
\end{proof}

\subsection{Matroid-based $(f,g)$-bounded packing of $k$ rooted trees}

The aim of this subsection is to extend Theorem \ref{thmKT} when we have two kinds of constraints on the roots of the rooted trees in the packing. In order to do so we need the following result that was introduced in \cite{book}   and proved in \cite[Theorem 12]{HSz5}.

\begin{thm}[\cite{book}, \cite{HSz5}]\label{genpart}
 Let $\{S_1,\dots,S_n\}$ be a partition of a  set $S$, $\alpha_i, \beta_i\in \mathbb Z_+$ for all $i\in\{1,\dots,n\}$ and $\mu\in \mathbb Z_+$.  Let 
\begin{eqnarray}
  \mathcal{B}	&	=	&	\{Z \subseteq S:Z\cap S_i\in \mathcal{I}_i, \alpha_i\leq |Z\cap S_i|\leq \beta_i \text{ for }i=1,\ldots,n,  |Z|=\mu\},\\
  	r(Z)		&	=	&	\min\{\sum_{i=1}^n \min \{\beta_i,|Z\cap S_i|\},\mu-\sum_{i=1}^n \max\{\alpha_i-|Z\cap S_i|,0\}\}
						\text{ for every } Z\subseteq S.\label{rfgpm}
\end{eqnarray}
There exists  a matroid whose  set of bases is $\mathcal{B}$ and  rank function is $r$ if and only if 
\begin{eqnarray}
	&&\hskip .6truecm\alpha_i\hskip .63truecm \leq \hskip .6truecm\min\{\beta_i, |S_i|\} \ \text{ for all } i\in\{1,\ldots,n\},			\label{genpartcond1}\\
	&&\sum_{i=1}^n \alpha_i \leq \mu \leq \sum_{i=1}^n \min\{\beta_i,|S_i|\}. \label{genpartcond2}
\end{eqnarray}
This matroid is called {\rm generalized partition matroid}.
\end{thm}

We also need the matroid intersection theorem of Edmonds \cite{e3}.

\begin{thm} [Edmonds \cite{e3}]\label{matroidintersection}
Let ${\sf M}_1=(S,r_1)$ and ${\sf M}_2=(S,r_2)$ be two matroids on  $S$, and $\mu\in\mathbb{Z}_+$.  A common independent set of size $\mu$  of ${\sf M}_1$ and ${\sf M}_2$ exists  if and only if 
\begin{eqnarray} 
	r_1(Z)+r_2(S\setminus Z)\geq \mu \ \ \ \text{ for all } Z \subseteq S.\label{edmmatint}
\end{eqnarray}
\end{thm}

We are now able to present and prove an extension of Theorem \ref{thmKT}.

\begin{thm}\label{hvkgcvkl}
Let $G=(V,E)$ be a graph, $S$ a multiset of vertices in $V$, $k\in\mathbb Z_+,$ $f,g: V\rightarrow \mathbb Z_+$ functions, and ${\sf M}=(S,r_{\sf M})$ a matroid. There exists an {\sf M}-based $(f,g)$-bounded packing of $k$ rooted trees  in $G$ if and only if  
\begin{eqnarray} 
 		f(v)	&	\le	&	\hskip .65truecm\min\{r_{\sf M}(S_v),g(v)\}\hskip 4.2truecm\text{ for every } v\in V,\label{newfgrm}\\
		k	&	\le	&	\sum_{v\in V}\min\{r_{\sf M}(S_v),g(v)\},\label{oygvouiph}\\
e_E({\cal P})	&	\ge	&	\sum_{X\in {\cal P}} \max\{r_{{\sf M}}(S)-r_{{\sf M}}(S_Y)-g(X\setminus Y):Y\subseteq X\} 
						\hskip .24truecm\text{ for every partition } {\cal P} \text{ of }  V,\label{rfrrgrgrt}\\
e_E({\cal P})+k	&	\ge	&	\sum_{X\in {\cal P}} \max\{r_{{\sf M}}(S)-r_{{\sf M}}(S_Y)+f(Y):Y\subseteq X\} 
						\hskip .9truecm\text{ for every partition } {\cal P} \text{ of }  V.\label{rfrrgrgrt2}\hskip .6truecm
\end{eqnarray}
\end{thm}

\begin{proof}
To prove the {\bf necessity}, let {\boldmath${\cal B}$} be an {\sf M}-based $(f,g)$-bounded packing of $k$ rooted  trees with root set $\bm{T}$. Since $r_{\sf M}$ is non-decreasing and ${\cal B}$ is {\sf M}-based, $r_{\sf M}(S_v)\ge r_{\sf M}(T_v)=|T_v|$ for every $v\in V.$ Then, since ${\cal B}$ is $(f,g)$-bounded, we have $\min\{r_{\sf M}(S_v),g(v)\}\ge |T_v|\ge f(v),$ so \eqref{newfgrm} holds. Further, we get $\sum_{v\in V}\min\{r_{\sf M}(S_v),g(v)\}\ge \sum_{v\in V}|T_v|=|T|=k,$ so \eqref{oygvouiph} holds.
Moreover, since ${\cal B}$ is {\sf M}-based, we get $r_{{\sf M}}(S)=r_{{\sf M}}(T).$ So,  by Theorem \ref{thmKT} applied for $T$, the submodularity and the subcardinality  of $r_{{\sf M}},$  we have 
\begin{eqnarray}\label{ouvjbkn}
	e_E({\cal P})	\ge		\sum_{X\in {\cal P}}          (r_{{\sf M}}(S)-r_{{\sf M}}(T_X))
				\ge		\sum_{X\in {\cal P}} \max\{r_{{\sf M}}(S)-r_{{\sf M}}(T_Y)-|T_{X\setminus Y}|:Y\subseteq X\}.
\end{eqnarray}
By \eqref{ouvjbkn}, the monotonicity of $r_{{\sf M}}$  and $g(v)\ge |T_v|$ for every $v\in V,$ we get that \eqref{rfrrgrgrt} holds.
By $|T_v|\ge f(v)$ for every $v\in V,$ we have $|T_{X\setminus Y}|\le |T_X|-f(Y)$ for every $Y\subseteq X\subseteq V.$ So, by \eqref{ouvjbkn}, the monotonicity of $r_{{\sf M}}$
 and $\sum_{X\in {\cal P}}|T_X|=|T|=k$,   \eqref{rfrrgrgrt2} holds.
\medskip

To prove the {\bf sufficiency}, let us suppose that \eqref{newfgrm}--\eqref{rfrrgrgrt2} hold. We may suppose that \eqref{matcondori1} holds. Indeed, by taking a maximal $S^*\subseteq S$ such that $S^*_v\in \mathcal{I}_{\sf M}$ for every $v\in V,$ we have $r_{\sf M}(S^*_X)=r_{\sf M}(S_X)$ for every $X\subseteq V,$ so \eqref{newfgrm}--\eqref{rfrrgrgrt2} still hold. From now on we suppose that \eqref{matcondori1} holds.
We formulate our problem as the intersection of two matroids. The first matroid is ${\sf M}'_{KT}$ with rank function $r'_{KT}$ given in \eqref{ZHfunction}. The second matroid ${\sf M}_2$ is the direct sum of the uniform matroid on ground set $E$ of rank $r_{{\sf M}}(S)|V|-k$ and the generalized partition matroid on  ground set $S$   with the partition $\{S_v\}_{v\in V}$ and the bounds $(\alpha_v,\beta_v)=(f(v),g(v))$ and $\mu=k.$ The rank function of ${\sf M}_2$ satisfies, by \eqref{rfgpm}, for all $F\subseteq E, T\subseteq S,$
\begin{eqnarray}
r_2(F\cup T)=\min\{|F|,r_{{\sf M}}(S)|V|-k\}+\min\{\sum_{v\in V} \min \{g(v),|T_v|\},k-\sum_{v\in V} \max\{f(v)-|T_v|,0\}\}.
\end{eqnarray}
Note that, by Theorem \ref{genpart} and $r_{{\sf M}}(S_v)=|S_v|$ for every $v\in V$ (by \eqref{matcondori1}),  the generalized partition matroid exists if and only if \eqref{newfgrm} and \eqref{oygvouiph} hold and $f(V)\le k$ (which holds by \eqref{rfrrgrgrt2}).

\begin{cl}\label{ouyvfouefh}
There exists an {\sf M}-based $(f,g)$-bounded packing of $k$ rooted trees with  edge set $F$ and root set $T$ if and only if $F\cup T$ is a common independent set of ${\sf M}'_{KT}$ and ${\sf M}_2$ of size $r_{{\sf M}}(S)|V|.$
\end{cl}

\begin{proof}
To prove the {\bf necessity}, let {\boldmath${\cal B}$} be an {\sf M}-based $(f,g)$-bounded packing of $k$ rooted trees with root set $T$. By Theorem \ref{thmKTimplicitgen}, $F\cup T$ is an  independent set of ${\sf M}'_{KT}$ of size $r_{{\sf M}}(S)|V|.$ Since ${\cal B}$ is {\sf M}-based, $T_v$ is independent in {\sf M} for every  $v\in V.$  Since ${\cal B}$ is an $(f,g)$-bounded packing of $k$ rooted trees, we have $f(v)\le  |T_v|\le g(v)$ and $|T|=k.$ Then $|F|=r_{{\sf M}}(S)|V|-k.$ It follows that $F\cup T$ is an independent set of  ${\sf M}_2$, and we are done.

To prove the {\bf sufficiency}, let us suppose that $F\cup T$ is a common independent set of ${\sf M}'_{KT}$ and ${\sf M}_2$ of size $r_{{\sf M}}(S)|V|.$ By Theorem \ref{thmKTimplicitgen}, there exists an {\sf M}-based packing ${\cal B}$ of rooted trees with edge set $F$ and root set $T$. Since $F\cup T$ is independent in ${\sf M}_2$ of size $r_{{\sf M}}(S)|V|,$ we have $|T|=k$ and $f(v)\le  |T_v|\le g(v)$ for every  $v\in V.$  Hence ${\cal B}$ is an $(f,g)$-bounded packing of $k$ rooted trees, and we are done.
\end{proof}

By Claim \ref{ouyvfouefh} and Theorem \ref{matroidintersection}, there exists an {\sf M}-based $(f,g)$-bounded packing of $k$ rooted trees if and only if 
\begin{eqnarray}\label{obdozibdoiz}
	\min\{r'_{KT}(F\cup T)+r_2(\overline F\cup\overline T):F\subseteq E, T\subseteq S\}\ge r_{{\sf M}}(S)|V|.
\end{eqnarray}
Let $F\subseteq E$ and $T\subseteq S$ attain the minimum.
\medskip

\noindent {\bf Case 1.} If $|\overline F|\ge r_{{\sf M}}(S)|V|-k.$ Then, $r_2(E\cup\overline T)=r_2(\overline F\cup\overline T).$ Hence, since  $r'_{KT}$ is non-decreasing, we have $r'_{KT}(T)+r_2(E\cup\overline T)\le r'_{KT}(F\cup T)+r_2(\overline F\cup\overline T)$, so we may suppose that $F=\emptyset.$
 By Claim \ref{kiytckhjfvlbnk}, we have 
\begin{eqnarray}
r'_{KT}(T)	&=&	|T|\ \ =\ \ \sum_{v\in V}|T_v|. \label{hjlbknkjhg}
\end{eqnarray}

\noindent {\bf Case (a)} If $r_2(E\cup\overline T)=r_{{\sf M}}(S)|V|-k+\sum_{v\in V} \min \{g(v),|\overline T_v|\}$, then, by \eqref{hjlbknkjhg}, \eqref{oygvouiph}, and \eqref{matcondori1}, we have 
\begin{eqnarray*}
r'_{KT}(T)+r_2(E\cup\overline T)-r_{{\sf M}}(S)|V|&=&\sum_{v\in V}|T_v|-k+\sum_{v\in V} \min \{g(v),|\overline T_v|\}\\
&\ge & -k+\sum_{v\in V} \min \{g(v),|S_v|\}\ \ \ge\ \ 0.
\end{eqnarray*}

\noindent {\bf Case (b)} If $r_2(E\cup\overline T)=r_{{\sf M}}(S)|V|-k+k-\sum_{v\in V} \max\{f(v)-|\overline T_v|,0\}$, then, by \eqref{hjlbknkjhg},  \eqref{newfgrm},  and \eqref{matcondori1}, we have 
\begin{eqnarray*}
r'_{KT}(T)+r_2(E\cup\overline T)-r_{{\sf M}}(S)|V|	&=&	     \sum_{v\in V}(|T_v|-\max\{f(v)-|\overline T_v|,0\})\\
										&=&   \sum_{v\in V}\min\{|T_v|,|S_v|-f(v)\}\ \ \ge\ \ 0.
\end{eqnarray*}

\noindent {\bf Case 2.} If $|\overline F|<r_{{\sf M}}(S)|V|-k.$ Then, $r_2(\overline T)=r_2(\overline F\cup\overline T)-|\overline F|.$ Hence, by the submodularity and the subcardinality of $r_{{\sf M}}$, we have $r'_{KT}(E\cup T)+r_2(\overline T)\le r'_{KT}(F\cup T)+|\overline F|+r_2(\overline F\cup\overline T)-|\overline F|$, so we may suppose that $F=E.$
\medskip

Note that $r'_{KT}(E\cup T)=r_{{\sf M}}(S)|V|+\min\{e_E(\mathcal{P})-\sum_{X\in {\cal P}} (r_{{\sf M}}(S)-r_{{\sf M}}(T_X)):\mathcal{P} \text{ partition of } V\}.$
\medskip

\noindent {\bf Case (a)} If $r_2(\overline T)=\sum_{v\in V} \min \{g(v),|\overline T_v|\}$. By  the modularity of $g$,  \eqref{matcondori1},  the monotonicity  and submodularity of $r_{{\sf M}}$, for every $X\subseteq V,$ there exists $\bm{Y_X}\subseteq X$ such that we have  
\begin{eqnarray}
r_{{\sf M}}(T_X)+\sum_{v\in X} \min \{g(v),|\overline T_v|\}&=&r_{{\sf M}}(T_X)+\sum_{v\in X\setminus Y_X}g(v)+\sum_{v\in Y_X}r_{{\sf M}}(\overline T_v)\label{yhjvbhvydtx}\\
&\ge& r_{{\sf M}}(T_{Y_X})+g(X\setminus Y_X)+r_{{\sf M}}(\overline T_{Y_X})\ \ \ge \ \ r_{{\sf M}}(S_{Y_X})+g(X\setminus Y_X).\nonumber
\end{eqnarray}
 Then, by \eqref{yhjvbhvydtx} and \eqref{rfrrgrgrt}, we have  
\begin{eqnarray*}
&&r'_{KT}(E\cup T)+r_2(\overline T)-r_{{\sf M}}(S)|V|\\
&=&\min\{e_E(\mathcal{P})+\sum_{X\in {\cal P}} (r_{{\sf M}}(T_X)-r_{{\sf M}}(S)+\sum_{v\in X} \min \{g(v),|\overline T_v|\}):\mathcal{P} \text{ partition of } V\}\\
&\ge& \min\{e_E(\mathcal{P})+\sum_{X\in {\cal P}} (r_{{\sf M}}(S_{Y_X})-r_{{\sf M}}(S)+g(X\setminus {Y_X})):\mathcal{P} \text{ partition of } V\}\ \ \ge \ \ 0.
\end{eqnarray*}

\noindent {\bf Case (b)} If $r_2(\overline T)=k-\sum_{v\in V} \max\{f(v)-|\overline T_v|,0\}.$ By  the modularity of $f$,  \eqref{matcondori1},  the monotonicity  and submodularity of $r_{{\sf M}}$, for every $X\subseteq V,$ there exists $\bm{Y_X}\subseteq X$ such that we have    
\begin{eqnarray}
r_{{\sf M}}(T_X)-\sum_{v\in X} \max\{f(v)-|\overline T_v|,0\}&=&r_{{\sf M}}(T_X)+\sum_{v\in Y_X}(r_{{\sf M}}(\overline T_v)-f(v))\nonumber\\
&\ge& r_{{\sf M}}(T_{Y_X})+r_{{\sf M}}(\overline T_{Y_X})-f(Y_X)\ \ \ge \ \  r_{{\sf M}}(S_{Y_X})-f(Y_X).
\hskip .6truecm\label{oiucudgfghj}
\end{eqnarray}
 Then, by \eqref{oiucudgfghj} and \eqref{rfrrgrgrt2}, we have  
 \begin{eqnarray*}
&&r'_{KT}(E\cup T)+r_2(\overline T)-r_{{\sf M}}(S)|V|\\
&=&\min\{e_E(\mathcal{P})+\sum_{X\in {\cal P}} (r_{{\sf M}}(T_X)-r_{{\sf M}}(S)-\sum_{v\in X} \max\{f(v)-|\overline T_v|,0\})+k:\mathcal{P} \text{ partition of } V\}\\
&\ge& \min\{e_E(\mathcal{P})+k +\sum_{X\in {\cal P}} (r_{{\sf M}}(S_{Y_X})-r_{{\sf M}}(S)-f(Y_X)):\mathcal{P} \text{ partition of } V\}\ \ \ge\ \ 0.
\end{eqnarray*}

It follows that in every case \eqref{obdozibdoiz} holds, and hence the required packing exists.
\end{proof}

If $f(v)=0$ and $g(v)=\infty$ for every $v\in V$ and $k=|S|,$ then Theorem \ref{hvkgcvkl} reduces to Theorem \ref{thmKT}.

\subsection{Matroid-based $(f,g)$-bounded $(\alpha,\beta)$-limited packing of  rooted trees}

Theorem \ref{hvkgcvkl} can easily be extended to packings where the number of rooted trees is not given but is lower   and upper bounded.

\begin{thm}\label{hvkgcvkl2}
Let $G=(V,E)$ be a graph, $S$ a multiset of vertices in $V$, $\alpha,\beta\in\mathbb Z_+,$ $f,g: V\rightarrow \mathbb Z_+$ functions, and ${\sf M}=(S,r_{\sf M})$ a matroid. There exists an {\sf M}-based $(f,g)$-bounded $(\alpha,\beta)$-limited packing of rooted trees  in $G$ if and only if  \eqref{newfgrm} and \eqref{rfrrgrgrt} hold and
\begin{eqnarray} 
			\alpha	&	\le	&	\beta,									\label{jbcdjhcvduvcy}\\
			\alpha	&	\le	&	\sum_{v\in V}\min\{r_{\sf M}(S_v),g(v)\},			\label{oygvouiphnew}\\
	e_E({\cal P})+\beta	&	\ge	&	\sum_{X\in {\cal P}} \max\{r_{{\sf M}}(S)-r_{{\sf M}}(S_Y)+f(Y):Y\subseteq X\} 
						\hskip .2truecm\text{ for every partition } {\cal P} \text{ of }  V.	\label{rfrrgrgrt2new}
\end{eqnarray}
\end{thm}

\begin{proof}
To prove the {\bf necessity}, let {\boldmath${\cal B}$} be an {\sf M}-based $(f,g)$-bounded $(\alpha,\beta)$-limited packing of rooted trees with root set $\bm{T}$. Let $k=|T|.$ Since ${\cal B}$ is $(\alpha,\beta)$-limited, we have $\alpha\le k\le\beta.$ Hence, \eqref{jbcdjhcvduvcy} holds. Further, by Theorem \ref{hvkgcvkl}, we get that \eqref{newfgrm}, \eqref{oygvouiph} (and hence \eqref{oygvouiphnew}),  \eqref{rfrrgrgrt} and \eqref{rfrrgrgrt2}  (and hence \eqref{rfrrgrgrt2new}) hold.
\medskip

To prove the {\bf sufficiency}, let us suppose that  \eqref{newfgrm}, \eqref{rfrrgrgrt}, \eqref{jbcdjhcvduvcy}, \eqref{oygvouiphnew}, and \eqref{rfrrgrgrt2new} hold. We show that there exists an integer $k$ that satisfies $\alpha\le k\le\beta,$ \eqref{oygvouiph} and \eqref{rfrrgrgrt2}. By \eqref{jbcdjhcvduvcy}, \eqref{oygvouiphnew}, and \eqref{rfrrgrgrt2new}, it is enough to prove that  for every partition  ${\cal P}$ of $V,$ we have 
\begin{eqnarray} \label{jkbhcd}
\sum_{v\in V}\min\{r_{\sf M}(S_v),g(v)\}&\ge &\sum_{X\in {\cal P}} \max\{r_{{\sf M}}(S)-r_{{\sf M}}(S_Y)+f(Y):Y\subseteq X\}-e_E({\cal P}).
\end{eqnarray}
 Let ${\cal P}$ be a partition of $V.$ For every $X\in {\cal P},$ let $\bm{Y_X, Y'_X}\subseteq X$ such that 
\begin{eqnarray} 
\sum_{v\in X}\min\{r_{\sf M}(S_v),g(v)\}				&	=	&	\sum_{v\in Y_X}r_{\sf M}(S_v)+g(X\setminus Y_X),\label{jigchjg}\\
\max\{r_{{\sf M}}(S)-r_{{\sf M}}(S_Y)+f(Y):Y\subseteq X\}	&	=	&	r_{{\sf M}}(S)-r_{{\sf M}}(S_{Y'_X})+f({Y'_X}). \label{jigchjg2}
\end{eqnarray}

Then, by \eqref{jigchjg}, \eqref{jigchjg2}, the submodularity of $r_{\sf M}$, the modularity of $g$ and $f$, \eqref{newfgrm} applied for $Y_X\cap Y'_X$ and for $Y'_X\setminus Y_X$, and \eqref{rfrrgrgrt}, we have 
\begin{eqnarray*} 
&&\sum_{v\in V}\min\{r_{\sf M}(S_v),g(v)\}-\sum_{X\in {\cal P}} \max\{r_{{\sf M}}(S)-r_{{\sf M}}(S_Y)+f(Y):Y\subseteq X\}\\ 
&=&\sum_{X\in {\cal P}}(\sum_{v\in Y_X}r_{\sf M}(S_v)+g(X\setminus Y_X)+r_{{\sf M}}(S_{Y'_X})-f({Y'_X})-r_{{\sf M}}(S))\\
&\ge&\sum_{X\in {\cal P}}(\sum_{v\in Y_X\cap Y'_X}r_{\sf M}(S_v)+r_{{\sf M}}(S_{Y_X\cup Y'_X})+g(X\setminus (Y_X\cup Y'_X))+g(Y'_X\setminus Y_X)\\
&&\hskip .8truecm -f(Y_X\cap Y'_X)-f(Y'_X\setminus Y_X)-r_{{\sf M}}(S))\\
&\ge&\sum_{X\in {\cal P}}(r_{{\sf M}}(S_{Y_X\cup Y'_X})+g(X\setminus (Y_X\cup Y'_X))-r_{{\sf M}}(S))\\
&\ge&\sum_{X\in {\cal P}}\min\{r_{{\sf M}}(S_Y)+g(X\setminus Y)-r_{{\sf M}}(S):Y\subseteq X\}\\
&=&-\sum_{X\in {\cal P}}\max\{r_{{\sf M}}(S)-r_{{\sf M}}(S_Y)-g(X\setminus Y):Y\subseteq X\}\ge -e_E({\cal P}),
\end{eqnarray*}
so \eqref{jkbhcd} holds. Then, by Theorem \ref{hvkgcvkl}, there exists an {\sf M}-based $(f,g)$-bounded packing of $k$ rooted trees  in $G$. Since $\alpha\le k\le\beta,$ the packing is $(\alpha,\beta)$-limited, and we are done.
\end{proof}

If $\alpha=\beta=k,$ then Theorem \ref{hvkgcvkl2} reduces to  Theorem \ref{hvkgcvkl}.

\subsection{Matroid-based $(f,g)$-bounded $(\alpha,\beta)$-limited packing of  rooted hypertrees}

Theorem \ref{tuttethmrem} was generalized to hypergraphs in  \cite{fkk}.

\begin{thm} [Frank, Kir\'aly, Kriesell \cite{fkk}]\label{tuttethmremhyp} 
Let $\mathcal{G}=(V,\mathcal{E})$ be a hypergraph and $k\in \mathbb{Z}_+.$
There exists a packing of $k$ spanning hypertrees in $\mathcal{G}$  if and only if  
	\begin{eqnarray} \label{tuttecondhyp} 
		e_{\mathcal{E}}({\cal P}) &\geq& k(|{\cal P}|-1)\hskip .5truecm  \text{ for every partition  ${\cal P}$ of  $V.$}
	\end{eqnarray}  
\end{thm}

If $\mathcal{G}$ is a graph, then Theorem \ref{tuttethmremhyp} reduces to Theorem \ref{tuttethmrem}.  In fact, Theorem \ref{tuttethmremhyp} can easily be proved  by Theorems \ref{trimminglemma} and \ref{tuttethmrem}. 
\medskip

We will now  exploit the fact that in Theorem \ref{trimminglemma} we can treat not only one but two supermodular functions on partitions.  We can hence generalize Theorem \ref{hvkgcvkl2} to hypergraphs.

\begin{thm}\label{hvkgcvkl2hyp}
Let $\mathcal{G}=(V,\mathcal{E})$ be a hypergraph, $S$ a multiset of vertices in $V$, $\alpha,\beta\in\mathbb Z_+,$ $f,g: V\rightarrow \mathbb Z_+$ functions, and ${\sf M}=(S,r_{\sf M})$ a matroid. There exists an {\sf M}-based $(f,g)$-bounded $(\alpha,\beta)$-limited packing of rooted hypertrees  in $\mathcal{G}$ if and only if   \eqref{newfgrm}, \eqref{jbcdjhcvduvcy}, and \eqref{oygvouiphnew} hold and
\begin{eqnarray} 
	e_\mathcal{E}({\cal P})	&	\ge	&	\sum_{X\in {\cal P}} \max\{r_{{\sf M}}(S)-r_{{\sf M}}(S_Y)-g(X\setminus Y):Y\subseteq X\} \hskip .2truecm\text{ for every partition } {\cal P} \text{ of }  V,\label{rfrrgrgrt2hyp}\\
	\beta+e_\mathcal{E}({\cal P})	&	\ge	&	\sum_{X\in {\cal P}} \max\{r_{{\sf M}}(S)-r_{{\sf M}}(S_Y)+f(Y):Y\subseteq X\} \hskip .8truecm\text{ for every partition } {\cal P} \text{ of }  V.\label{rfrrgrgrt2newhyp}\hskip .8truecm
\end{eqnarray}
\end{thm}

\begin{proof}
To prove the {\bf necessity}, suppose that  there exists an {\sf M}-based $(f,g)$-bounded $(\alpha,\beta)$-limited packing of rooted hypertrees  in $\mathcal{G}$. Then, by definition,  the hypertrees in the packing can be trimmed to get an {\sf M}-based $(f,g)$-bounded $(\alpha,\beta)$-limited packing of rooted trees. Then, by Theorem \ref{hvkgcvkl2}, we get that  \eqref{newfgrm}, \eqref{rfrrgrgrt}, \eqref{jbcdjhcvduvcy}, \eqref{oygvouiphnew}, and \eqref{rfrrgrgrt2new} hold. Then  \eqref{rfrrgrgrt2hyp} and \eqref{rfrrgrgrt2newhyp} hold in $\mathcal{G}$.

To prove the {\bf sufficiency}, suppose that  \eqref{newfgrm}, \eqref{jbcdjhcvduvcy}, \eqref{oygvouiphnew}, \eqref{rfrrgrgrt2hyp},    and \eqref{rfrrgrgrt2newhyp} hold. 
Note that \eqref{rfrrgrgrt2hyp}  and \eqref{rfrrgrgrt2newhyp} are equivalent to $e_\mathcal{E}({\cal P})\ge p_1(\mathcal{P})$ and $e_\mathcal{E}({\cal P})\ge p_2(\mathcal{P})$ for every partition $\mathcal{P}$ of $V,$ where the functions $p_1$ and $p_2$ are defined in \eqref{defpg} and \eqref{defpf}. By Claim \ref{icvhjbkn}, $p_1$ and $p_2$ are supermodular on partitions of $V.$
Thus, by Theorem \ref{trimminglemma}, we get a graph $G$ that satisfies \eqref{rfrrgrgrt} and \eqref{rfrrgrgrt2new}. Since  \eqref{newfgrm}, \eqref{jbcdjhcvduvcy}, and \eqref{oygvouiphnew} hold by assumption, Theorem \ref{hvkgcvkl2} implies that there exists  an {\sf M}-based $(f,g)$-bounded $(\alpha,\beta)$-limited packing of rooted trees  in $G$. By replacing each edge of $G$ by the hyperedge that was trimmed to it, we obtain   an {\sf M}-based $(f,g)$-bounded $(\alpha,\beta)$-limited packing of rooted hypertrees  in $\mathcal{G}$.
\end{proof}

If $\mathcal{G}$ is a graph, then Theorem \ref{hvkgcvkl2hyp} reduces to Theorem \ref{hvkgcvkl2}. If $S$ is a multiset of vertices in $V$ of size $k$, ${\sf M}$ is the free matroid on $S$, $\alpha=\beta=k,$ $f(v)=0$ and $g(v)=\infty$ for every $v\in V$, then Theorem \ref{hvkgcvkl2hyp} reduces to Theorem \ref{tuttethmremhyp}.

\subsection{Augmentation for matroid-based $(f,g)$-bounded $(\alpha,\beta)$-limited packing of  rooted hypertrees}

Frank \cite{edgeconn} solved the augmentation version of Theorem \ref{tuttethmrem} in which a minimum number of edges must be added to a graph to have a packing of $k$ spanning trees.

\begin{thm}[Frank \cite{edgeconn}]\label{augspan}
Let ${G}=(V,{E})$ be a graph and $k,\gamma \in\mathbb Z_+,$ We can add $\gamma$ edges to $G$ to have a packing of $k$ spanning trees if and only if
\begin{eqnarray} 
\gamma+e_{E}({\cal P})&\ge&k(|{\cal P}|-1).
\end{eqnarray}
\end{thm}

If $\gamma=0,$ then Theorem \ref{augspan} reduces to Theorem \ref{tuttethmrem}. Theorem \ref{augspan} can be easily proved by Theorem \ref{coveringpartsuper12} applied for $p_1(\mathcal{P})=p_2(\mathcal{P})=k(|{\cal P}|-1)-e_{E}({\cal P}).$ Note that, by Claim \ref{lkbzlkbl1}(c) and Lemma \ref{submodeap}, $p_1=p_2$ is a supermodular function on partitions. 
\medskip

We will now  exploit the fact that in Theorem \ref{coveringpartsuper12} we can treat two different supermodular functions on partitions.  We can hence  propose a common generalization of Theorems \ref{hvkgcvkl2hyp} and \ref{augspan}.

\begin{thm}\label{kjvjhcychjkjhyp}
Let $\mathcal{G}=(V,\mathcal{E})$ be a hypergraph, $S$ a multiset of vertices in $V$, $\alpha,\beta,\gamma\in\mathbb Z_+,$ $f,g: V\rightarrow \mathbb Z_+$ functions, and ${\sf M}=(S,r_{\sf M})$ a matroid. We can add $\gamma$ edges to $\mathcal{G}$ to have an {\sf M}-based $(f,g)$-bounded $(\alpha,\beta)$-limited packing of rooted hypertrees if and only if \eqref{newfgrm}, \eqref{jbcdjhcvduvcy}, and \eqref{oygvouiphnew}   hold and
\begin{eqnarray} 
r_{{\sf M}}(S)-r_{{\sf M}}(S_Y) &\le&\min\{\beta-f(Y),g(\overline Y)\}\hskip 4.25truecm\text{ for every } Y\subseteq V,\label{rmgrms}\\
	g(V)+\gamma+e_\mathcal{E}({\cal P})	&	\ge	&	\sum_{X\in {\cal P}} \max\{r_{{\sf M}}(S)-r_{{\sf M}}(S_Y)+g(Y):Y\subseteq X\} \text{ for every partition } {\cal P} \text{ of }  V,\label{rfrrgrgrtaughyp}\\
	\beta+\gamma+e_\mathcal{E}({\cal P})	&	\ge	&	\sum_{X\in {\cal P}} \max\{r_{{\sf M}}(S)-r_{{\sf M}}(S_Y)+f(Y):Y\subseteq X\} \text{ for every partition } {\cal P} \text{ of }  V.\hskip .6truecm\label{rfrrgrgrt2newaughyp}
\end{eqnarray}
\end{thm}

\begin{proof}
To prove the {\bf necessity}, suppose that  we can add an edge set $F$ of size $\gamma$ to $\mathcal{G}$ to have an {\sf M}-based $(f,g)$-bounded $(\alpha,\beta)$-limited packing of rooted hypertrees in $\mathcal{G}+F=(V,\mathcal{E}')$. Then, by Theorem \ref{hvkgcvkl2hyp}, we get that \eqref{newfgrm}, \eqref{jbcdjhcvduvcy},  \eqref{oygvouiphnew}, \eqref{rfrrgrgrt2hyp}, and \eqref{rfrrgrgrt2newhyp} hold for $\mathcal{E}'$. Applying \eqref{rfrrgrgrt2hyp} and \eqref{rfrrgrgrt2newhyp} for ${\cal P}=\{V\},$ we get \eqref{rmgrms}. Since $e_{\mathcal{E}'}({\cal P})\le e_\mathcal{E}({\cal P})+\gamma,$ \eqref{rfrrgrgrt2hyp} and \eqref{rfrrgrgrt2newhyp} imply \eqref{rfrrgrgrtaughyp} and \eqref{rfrrgrgrt2newaughyp}.

To prove the {\bf sufficiency}, suppose that  \eqref{newfgrm}, \eqref{jbcdjhcvduvcy}, \eqref{oygvouiphnew}, \eqref{rmgrms},   \eqref{rfrrgrgrtaughyp},  and \eqref{rfrrgrgrt2newaughyp} hold. Let the functions {\boldmath$p'_1$} and {\boldmath$p'_2$} be defined as follows. For every partition $\mathcal{P}$ of $V,$ $p'_1({\cal P})=p_1({\cal P})-e_\mathcal{E}({\cal P})$ and $p'_2({\cal P})=p_2({\cal P})-e_\mathcal{E}({\cal P})$, where $p_1$ and $p_2$ are defined in \eqref{defpg} and \eqref{defpf}. By Claim \ref{icvhjbkn} and Lemma \ref{submodeap}, $p'_1$ and $p'_2$ are supermodular on partitions of $V.$ By \eqref{rmgrms}, we get that \eqref{partcovcond1} holds for $p'_1$ and $p'_2$. By \eqref{rfrrgrgrtaughyp} and \eqref{rfrrgrgrt2newaughyp}, we get that \eqref{partcovcond2} holds for $p'_1$ and $p'_2$. Hence  Theorem \ref{coveringpartsuper12} implies that there exists an edge set $F$ on $V$ of size $\gamma$ such that $e_F(\mathcal{P})\ge\max\{p'_1(\mathcal{P}), p'_2(\mathcal{P})\}$ for every partition $\mathcal{P}$ of $V.$ This means that in the hypergraph $\mathcal{G}'=(V,\mathcal{E}'=\mathcal{E}\cup F)$, \eqref{rfrrgrgrt2hyp} and \eqref{rfrrgrgrt2newhyp} hold for $\mathcal{E}'$. Since  \eqref{newfgrm}, \eqref{jbcdjhcvduvcy}, and \eqref{oygvouiphnew} also hold, by Theorem \ref{hvkgcvkl2hyp}, there exists an {\sf M}-based $(f,g)$-bounded $(\alpha,\beta)$-limited packing of rooted hypertrees  in $\mathcal{G}',$ which completes the proof of Theorem \ref{kjvjhcychjkjhyp}.
\end{proof}

If $\gamma=0,$ then Theorem \ref{kjvjhcychjkjhyp} reduces to Theorem \ref{hvkgcvkl2hyp}. If $\mathcal{G}$ is a graph, $S$ is a multiset of vertices in $V$ of size $k$ and ${\sf M}$ is the free matroid on $S$,  $f(v)=0$ and $g(v)=\infty$ for every $v\in V$ and $\alpha=\beta=k,$ then Theorem \ref{kjvjhcychjkjhyp} reduces to \ref{augspan}.
 
\subsection{Augmentation for matroid-based $(f,g)$-bounded $(\alpha,\beta)$-limited packing of  rooted  hyperforests}

The argument  in \cite{HSz4} showing that Theorem 4  in \cite{HSz4} implies Theorem 5  in \cite{HSz4} can be applied here as well. Hence  Theorem \ref{kjvjhcychjkjhyp} implies the following result.

\begin{thm}\label{kjvjhcychjkjforhyp}
Let $\mathcal{G}=(V,\mathcal{E})$ be a hypergraph, ${\cal S}$ a family of subsets of $V$,   $\alpha,\beta,\gamma\in\mathbb Z_+,$ $f,g: V\rightarrow \mathbb Z_+$ functions, and ${\sf M}=({\cal S},r_{\sf M})$ a matroid. We can add $\gamma$ edges to $\mathcal{G}$ to have an {\sf M}-based $(f,g)$-bounded $(\alpha,\beta)$-limited packing of rooted hyperforests if and only if \eqref{jbcdjhcvduvcy} and \eqref{rmgrms} hold and
\begin{eqnarray} 
	f(v)	&	\le	&\hskip .63truecm  \min\{r_{\sf M}({\cal S}_v),g(v)\}\hskip 3.4truecm	\text{ for every } v\in V,\label{newfgrmbr}\\
\alpha	&	\le	&	\sum_{v\in V}\min\{r_{\sf M}({\cal S}_v),g(v)\},		\label{oygvouiphnewfor}\\
	g(V)+\gamma+e_\mathcal{E}({\cal P})		&	\ge	&	\sum_{X\in {\cal P}} \max\{r_{{\sf M}}({\cal S})-r_{{\sf M}}({\cal S}_Y)+g(Y):Y\subseteq X\} \hskip .14truecm\text{ for every partition } {\cal P} \text{ of }  V,\label{rfrrgrgrtaugforhyp}\\
	\beta+\gamma+e_\mathcal{E}({\cal P})	&	\ge	&	\sum_{X\in {\cal P}} \max\{r_{{\sf M}}({\cal S})-r_{{\sf M}}({\cal S}_Y)+f(Y):Y\subseteq X\} \hskip 0.1truecm\text{ for every partition } {\cal P} \text{ of }  V.\hskip .7truecm\label{rfrrgrgrt2newaugforhyp}
\end{eqnarray}
\end{thm}

If ${\cal S}=\{S_v\}_{v\in V}$, then Theorem \ref{kjvjhcychjkjforhyp} reduces to Theorem \ref{kjvjhcychjkjhyp}. 
Note that Theorem \ref{kjvjhcychjkjforhyp} implies all the results of this section.
\medskip

The case  $\alpha=0$ and $\beta=\infty$ of Theorem \ref{kjvjhcychjkjforhyp} coincides with a special case of a result of \cite{szighypbranc}.

\end{document}